\documentclass[11pt,reqno]{amsart}
\usepackage{amscd,amssymb,amsmath,amsthm}
\usepackage[arrow,matrix]{xy}
\usepackage{graphicx}
\usepackage{epstopdf}
\usepackage{color}
\newtheorem{thm}{Theorem}
\newtheorem{defn}{Definition}
\newtheorem{lemma}{Lemma}
\newtheorem{pro}{Proposition}

\newtheorem{cor}{Corollary}

\numberwithin{equation}{section} 

\begin{document}

\title[Population dynamical of mosquito]{A discrete-time dynamical system and an evolution algebra of mosquito population}

\author{U. A. Rozikov, M. V. Velasco}

 \address{U.\ A.\ Rozikov\\ Institute of mathematics,
81, Mirzo Ulug'bek str., 100125, Tashkent, Uzbekistan.}
\email {rozikovu@yandex.ru}

\address{M.\ V. \ Velasco\\ Departamento de An\'{a}lisis Matem\'{a}tico
Facultad de Ciencias Universidad de Granada
18071- Granada, Spain.}
\email {vvelasco@ugr.es}

\begin{abstract} Recently,  continuous-time dynamical systems,
based on systems of ordinary differential equations,
for mosquito populations are studied.  In this paper we
consider discrete-time dynamical system generated by an evolution
quadratic operator of mosquito population and show that this system
has two fixed points, which are saddle points (under some conditions on
the parameters of the system).  We construct an evolution algebra taking
its matrix of structural constants equal to the Jacobian of the
quadratic operator at a fixed point. Idempotent and absolute nilpotent elements,
simplicity properties and some limit points of the evolution operator
corresponding to the evolution algebra are studied. We give some biological interpretations of our results.
\end{abstract}
\maketitle

{\bf{Key words.}}
Mathematical model, Mosquito dispersal, Discrete-time, fixed point, limit point.

{\bf Mathematics Subject Classifications (2010).}  92D25 (34C60 34D20 92D30 92D40)

\section{Introduction} \label{sec:intro}

In this paper we consider a discrete-time dynamical system which is related to
a mosquito population dynamics. This is a six-dimensional nonlinear (quadratic)
dynamical system. Let us give necessary definitions first and then
we will give some facts and formulate our problems.

{\it A model of mosquito dispersal.} Following \cite{LP} we give a mathematical model of mosquito dispersal.
It is known\footnote{Source http://www.mosquito.org/page/lifecycle} that all mosquito species life cycle has four distinct stages:
\begin{itemize}
\item[(i)] Egg (denote the variable by $E$)- hatches when exposed to water.
\item[(ii)] Larva (denoted by $L$)- the stage after eggs hatch, lives in water.
\item[(iii)] Pupa ($P$)- ``tumbler" does not feed; stage just before emerging as adult.
\item[(iv)] Adult ($A$) - flies short time after emerging and after its body parts have hardened.
\end{itemize}

The stages (i)-(iii) occur in water, but the adult is an active flying insect.
Only the female mosquito bites and feeds on the blood of humans or other animals.
The female mosquito obtains a blood meal, and after that lays the eggs directly on or near water.
The eggs can survive dry conditions for a few months.
The eggs hatch in water and a mosquito larva or ``wriggler" emerges. The length of time to hatch
depends on water temperature, food and type of mosquito.
The larva lives in the water, feeds and develops into the third stage of the life
cycle called, a pupa or ``tumbler." The pupa also lives in the water but no longer feeds.
Finally, the mosquito emerges from the pupal case after two days to a week in the pupal stage.
The life cycle typically takes up two weeks, but depending on conditions, it can range
from 4 days to as long as a month.
The adult mosquito emerges onto the water's surface and flies away, ready to begin its life cycle.
The adults are classified as host seeking adults ($A_h$), resting adults ($A_r$), and oviposition or breeding site
seeking adults ($A_o$).
Thus at time moment $t\geq 0$ the state of the population is given by the density vector $(E(t), L(t), P(t), A_h(t), A_r(t), A_o(t))$.
\begin{figure}
  \includegraphics[width=8cm]{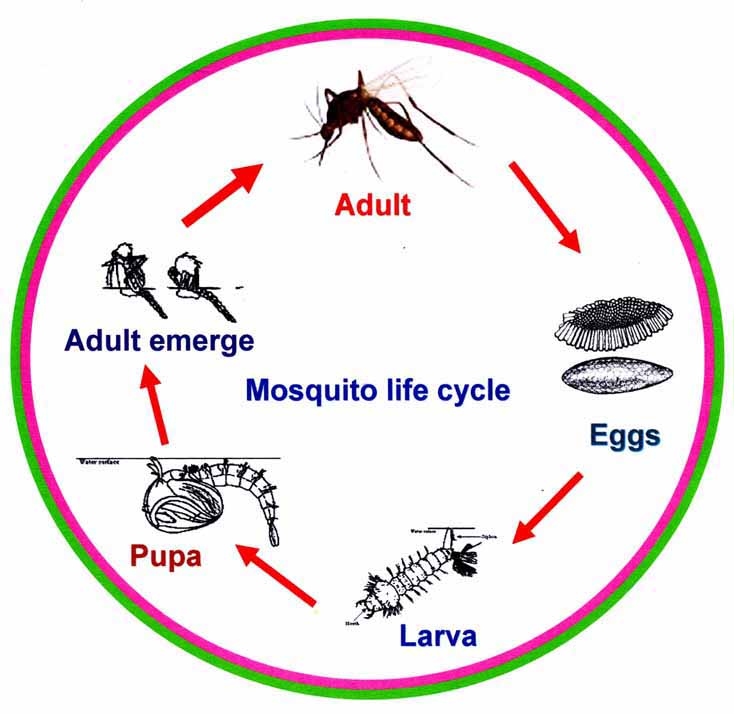}\\
  \caption{The life cycle of mosquito. Source: internet. }\label{fig1}
\end{figure}
The following system of
differential equations describes mosquito dynamics (see \cite{LP}):
\begin{equation}\label{ds}
\begin{array}{llllll}
{dE\over dt}=b\rho_{A_o}A_o-(\mu_E+\rho_E)E,\\[2mm]
{dL\over dt}=\rho_{E}E-(\mu_{1L}+\mu_{2L}L+\rho_L)L,\\[2mm]
{dP\over dt}=\rho_{L}L-(\mu_P+\rho_P)P,\\[2mm]
{dA_h\over dt}=\rho_{P}P+\rho_{A_o}A_o-(\mu_{A_h}+\rho_{A_h})A_h,\\[2mm]
{dA_r\over dt}=\rho_{A_h}A_h-(\mu_{A_r}+\rho_{A_r})A_r,\\[2mm]
{dA_o\over dt}=\rho_{A_r}A_r-(\mu_{A_o}+\rho_{A_o})A_o,
\end{array}
\end{equation}
with initial conditions  $E(0), L(0), P(0), A_h(0), A_r(0), A_o(0)$,
the description and values of parameters of this system of equations
are given in the following table (see \cite{LP} and references therein):

\begin{tabular}{|l|l|l|l|}
\hline
Parameter & Description & Units & Range\\
\hline
$b$& number of female eggs laid per &-&50 - 300\\
& oviposition&&\\
\hline
$\rho_E$& egg hatching rate & day$^{-1}$& 0.33 - 1.0\\
&into larvae&&\\
\hline
$\rho_L$& rate at which larvae develop & day$^{-1}$&0.08-0.17\\
&into pupae &&\\
\hline
$\rho_P$& rate at which pupae develop & day$^{-1}$&0.33 - 1.0\\
&into adult/emergence rate&&\\
\hline
$\mu_E$& egg mortality rate& day$^{-1}$&0.32-0.80\\
\hline
$\mu_{1L}$& density-independent larvae & day$^{-1}$&0.30-0.58\\
&mortality rate&&\\
\hline
$\mu_{2L}$& density-dependent larvae & day$^{-1}$, mosq.$^{-1}$& 0.0 - 1.0\\
&mortality rate&&\\
\hline
$\mu_P$& pupae mortality rate& day$^{-1}$&0.22-0.52\\
\hline
$\rho_{A_h}$& rate at which host seeking& day$^{-1}$& 0.322 - 0.598\\
& mosquitoes enter the resting state&&\\
\hline
$\rho_{A_r}$& rate at which resting mosquitoes & day$^{-1}$&0.30 -  0.56\\
&enter oviposition site searching state&&\\
\hline
$\rho_{A_o}$& oviposition rate& day$^{-1}$& 3.0-4.0\\
\hline
$\mu_{A_h}$& mortality rate of mosquitoes of & day$^{-1}$& 0.125 - 0.233\\
&searching for hosts&&\\
\hline
$\mu_{A_r}$& mortality rate of resting & day$^{-1}$& 0.0034 - 0.01\\
&mosquitoes&&\\
\hline
$\mu_{A_o}$& mortality rate of mosquitoes & day$^{-1}$&0.41 - 0.56\\
&searching for oviposition sites&&\\
\hline
\end{tabular}\\

{\bf Table 1}. The description and values of parameters. \\

In \cite{LP} the following results were proved

\begin{thm} The continuous-time dynamical system (\ref{ds}) has the following properties \begin{itemize}
\item If $(E(0), L(0), P(0), A_h(0), A_r(0), A_o(0))\in \mathbb R^6_+$
then the system (\ref{ds}) has a unique solution
$(E(t), L(t), P(t), A_h(t), A_r(t), A_o(t))\in \mathbb R^6_+$,
for and $t\geq 0$.
\item There are two equilibrium  points $P_0=(0,0,0,0,0,0)$ (called mosquito-free)
and $P_e=(E^*,L^*,P^*,A_h^*,A_r^*,A^*_0)$ (called persistent).
\item The mosquito-free (resp. persistent) equilibrium is locally stable when $R_0<1$ (resp.
$R_0>1$) and unstable if $R_0>1$ (resp.
$R_0<1$), where $R_0$ is population reproduction number (which depends on all parameters of the model).
\end{itemize}
\end{thm}
We would like to construct and study discrete-time dynamical systems and evolution algebras
corresponding to the system of differential equations (\ref{ds}).

The paper is organized as follows. In Section 2 we consider discrete-time dynamical
system generated by an evolution quadratic operator of mosquito population
and show that this system has two fixed points, which are saddle points (under some conditions on
the parameters of the system). Section 3 is devoted to evolution algebras corresponding
to mosquito population. We construct such an algebra taking its matrix of structural constants
equal to the Jacobian of the quadratic operator considered
in Section 2 at a fixed point of the operator. We study idempotent and absolute nilpotent elements,
simplicity properties and some limit points of the evolution operator corresponding to the evolution algebra.
In the last section we give some biological interpretations of the results.

\section{Discrete-time dynamics of mosquito populations}

Investigating population dynamics of mosquitos one wants to have
a good control on this population, at least to make sure that
there is a limit value for the population.
This assumption reminds very famous example
of a limited population given in book \cite[pages 5-6]{De}:
The logistic map (a quadratic mapping),
often cited as an archetypal example of how complex,
chaotic behaviour can arise from
very simple non-linear discrete-time dynamical system.
But the continuous-time analog of
this quadratic dynamical system has no cyclic behavior
or other fluctuations in the population.
This example, showing that continuous and discrete time dynamics
of the same system may be essentially different,  motivates
us to consider discrete-time version of the system
(\ref{ds}).

To simplify notations we denote
\begin{equation}\label{sim}
\begin{array}{llll}
H=A_h, \, R=A_r, \, O=A_o, \, \theta=\rho_{O}, \, e=\rho_{E},\\
a=\rho_{L}, \, p=\rho_{P}, \, h=\rho_{H},\, r=\rho_{R},\\
\hat e=\mu_E+\rho_E, \, \hat l_1=\mu_{1L}+\rho_L, \, \hat l_2=\mu_{2L}, \, \hat p=\mu_P+\rho_P,\\
\hat h=\mu_H+\rho_H, \, \hat r=\mu_R+\rho_R, \, \hat\theta=\mu_{O}+\rho_{O}.
\end{array}
\end{equation}
Then discrete-time version of (\ref{ds}) has the following form
 \begin{equation}\label{dsd}
\begin{array}{llllll}
E_{n+1}=b \theta O_n+(1-\hat e)E_n,\\[2mm]
L_{n+1}=e E_n+\left(1-\hat l_1-\hat l_2L_n\right)L_n,\\[2mm]
P_{n+1}=a L_n+(1-\hat p)P_n,\\[2mm]
H_{n+1}=p P_n+\theta O_n+(1-\hat h)H_n,\\[2mm]
R_{n+1}=h H_n+(1-\hat r)R_n,\\[2mm]
O_{n+1}=r R_n+(1-\hat \theta)O_n,
\end{array}
\end{equation}
where $n$ is non-negative integer number and
$$(E_n, L_n, P_n, H_n, R_n, O_n)=(E(n), L(n), P(n), H(n), R(n), O(n)).$$
Consider operator $M: v=(E,L,P,H,R,O)\in \mathbb R^6_+\to v'=M(v)=(E',L',P',H',R',O')\in\mathbb R^+_6$ defined by
 \begin{equation}\label{dso}M: \left\{
\begin{array}{llllll}
E'=b \theta O+(1-\hat e)E,\\[2mm]
L'=e E+\left(1-\hat l_1\right)L-\hat l_2L^2,\\[2mm]
P'=a L+(1-\hat p)P,\\[2mm]
H'=p P+\theta O+(1-\hat h)H,\\[2mm]
R'=h H+(1-\hat r)R,\\[2mm]
O'=r R+(1-\hat \theta)O,
\end{array}
\right.
\end{equation}
by Table 1 and (\ref{sim}) we can see that some parameters of the operator (\ref{dso}) are strongly positive, but others may be negative and non-negative too.

Then the dynamical system can be written as $v_{n+1}=M(v_n)=\underbrace{M(M(\dots M(v_0))\dots)}_{n \, {\rm times}}$, $n\geq 0$.
We are interested to investigate the limit $\lim_{n\to \infty} v_n$, for any initial condition $v_0\in \mathbb R^6_{+}.$

{\bf Remark 1.} Let us give two important remarks about the discrete-time dynamical systems generated
  by the quadratic mapping (\ref{dso}) assuming that all parameters are non-negative:
  \begin{itemize}
\item[1.] If $e$ is very close to zero, then for the second coordinate, i.e.,
$L$, one can give a new scale to make the second coordinate as $L'\approx \lambda L(1-L)$,
for $\lambda=\hat l_2^{-1}(1-\hat l_1)^2$.
This means that in case of small $e$ the dynamics of the second
coordinate is close to the logistic quadratic map, which has very complex dynamics.

\item[2.] One can also choose parameters such that two coordinates of the map $M$
will be close to the H\'enon map, $H:\mathbb R^2\to \mathbb R^2$, defined by
$$H: \left\{\begin{array}{ll}
x'=1+y-ax^2\\[2mm]
y'=bx, \ \ \ \ a,b>0.
\end{array}\right.$$
which is the analogue of the quadratic map in dimension
two \cite[page 251]{De}. It is known that for some values of its parameters, the dynamics of
the H\'enon map is very complex, having infinitely many periodic points. This is one of the
most studied examples of dynamical systems that exhibit chaotic behavior.
\end{itemize}

\subsection{Fixed points of $M$.} A fixed point $v$ is solution to $v=M(v)$,
i.e. the solution of the following system
\begin{equation}\label{dsf}\left\{
\begin{array}{llllll}
E=b \theta O+(1-\hat e)E,\\[2mm]
L=e E+\left(1-\hat l_1\right)L-\hat l_2L^2,\\[2mm]
P=a L+(1-\hat p)P,\\[2mm]
H=p P+\theta O+(1-\hat h)H,\\[2mm]
R=h H+(1-\hat r)R,\\[2mm]
O=r R+(1-\hat \theta)O,
\end{array}\right.
\Rightarrow
\left\{
\begin{array}{llllll}
b \theta O-\hat eE=0,\\[2mm]
e E-\hat l_1L-\hat l_2L^2=0,\\[2mm]
a L-\hat pP=0,\\[2mm]
p P+\theta O-\hat hH=0,\\[2mm]
h H-\hat rR=0,\\[2mm]
r R-\hat \theta O=0.
\end{array}\right.
\end{equation}

The following proposition describes all fixed points.

\begin{pro}\label{p1} The operator $M$ has two fixed points:
$$P_0=(0,0,0,0,0,0), \ \ \ P_1=(E^*, L^*, P^*, H^*, R^*, O^*),$$
where
$$E^*={b\theta\over \hat e}O^*, \ \ L^*=CO^*, \ \ P^*={a\over \hat p}CO^*, \ \ H^*={\hat r\over r}{\hat \theta\over h}O^*, \ \ R^*={\hat \theta \over r}O^*,$$
$$C={\theta\over a}{\hat p\over p}\left({\hat h\over h} {\hat r\over r} {\hat\theta\over \theta}-1\right), \ \
O^*={eb\theta-\hat e\hat l_1 C\over \hat l_1 \hat l_2 C^2}.$$
\end{pro}
\begin{proof} Since equations of the system (\ref{dsf}) are linear (except the second one),
the proof consists simple computations.
\end{proof}

Now we shall examine the type of the fixed points.

\begin{defn}\label{d2} (see \cite{De}). A fixed point $v$ of an operator $M$ is called hyperbolic if
its Jacobian $J$ at $v$ has no eigenvalues on the unit circle.

 A hyperbolic fixed point $v$ is called:
\begin{itemize}
\item attracting if all the eigenvalues of the Jacobi matrix $J(v)$ are less than 1 in
absolute value;
\item repelling if all the eigenvalues of the Jacobi matrix $J(v)$ are greater than 1 in
absolute value;
\item a saddle point otherwise.
\end{itemize}
\end{defn}

To find the type of a fixed point of the operator (\ref{dso}) we write
the Jacobi matrix:
\begin{equation}\label{JM}
J(v)=J_M(v)=\left(\begin{array}{cccccc}
1-\hat e&0&0&0&0&b\theta\\[3mm]
e&1-\hat l_1-2\hat l_2 L&0&0&0&0\\[3mm]
0&a&1-\hat p&0&0&0\\[3mm]
0&0&p&1-\hat h&0&\theta\\[3mm]
0&0&0&h&1-\hat r&0\\[3mm]
0&0 &0&0&r&1-\hat \theta
\end{array}\right).
\end{equation}

The equation for eigenvalues is $\det(J(v)-\lambda I)=0$
which is\footnote{This equation is given in \cite[page 201]{LP} too, but there the non-zero term $- abehpr\theta$ is missed.}
$$\det(J(v)-\lambda I)=$$
\begin{equation}\label{ee}
(1-\hat e-\lambda)(1-\hat l_1-2\hat l_2L-\lambda)(1-\hat p-\lambda)[(1-\hat h-\lambda)(1-\hat r-\lambda)(1-\hat\theta-\lambda)+hr\theta]-abehpr\theta=0.\end{equation}
This is a polynomial equation of order 6, and can not be solved, in general. But one can have numerical solutions
by using Maple, for concrete given parameters $a, b, e, \hat e, h, \hat h$, $p, \hat p, r, \hat r, \theta, \hat \theta, \hat l_1, \hat l_2, \epsilon$, therefore we do not list them here.
After having some values of eigenvalues one can give condition on absolute
values of them to satisfy Definition \ref{d2}.

For example, we take `baseline' values of parameters mentioned in \cite{LP}:
$$a=0.14,\ \ b=100, \ \ e=0.5, \ \ \hat e=1.06, \ \ h=0.46, \ \ \hat h=0.64, \ \ p=0.5, \ \ \hat p=0.87,$$
\begin{equation}\label{base}
 r=0.43, \ \ \hat r=0.4343, \ \ \theta=3, \ \ \hat \theta=3.41, \ \ \hat l_1=0.58, \ \ \hat l_2=0.05.
 \end{equation}
Then for the case $L=0$ the equation (\ref{ee}) has the form:
$$(0.06+\lambda)(0.42-\lambda)(0.13-\lambda)[(0.36-\lambda)(0.5657-\lambda)(2.41+\lambda)-0.5934]-2.0769=0.$$
  Maple gives the following six solutions of the last equation:
 $$\lambda_1=1.256222386, \ \ \lambda_2=-2.35164464,$$
 $$\lambda_3=0.611373493+0.7919408816i, \ \ \lambda_4=0.611373493-0.7919408816i,$$
 $$ \lambda_5= -0.5608123659+0.6228748264i, \ \ \lambda_6=-0.5608123659-0.6228748264i.$$
 We have
 $$ |\lambda_1|=1.256222386, \ \ |\lambda_2|=2.35164464, \ \  |\lambda_{3}|=|\lambda_4|=1.000947908,$$
 $$|\lambda_5|=|\lambda_6|=0.7024835591.$$
 Therefore for the baseline parameters the fixed point $P_0$ is a saddle point.

The case of $L=L^*$: (see Proposition \ref{p1}), for baseline parameters we have $L^*= 209.2580821$.
Putting this value in the equation (\ref{ee}) (in case of baseline parameters) we obtain the following equation
$$(0.06+\lambda)(20.50580821+\lambda)(0.13-\lambda)[(0.36-\lambda)(0.5657-\lambda)(2.41+\lambda)-0.5934]+2.0769=0.$$
Using Maple we get the following six solutions:
$$\lambda_1=0.936104284,  \ \ \lambda_2=  -0.3491645453,$$
$$\lambda_3= -2.331267091, \ \ \lambda_4= -20.50580883,$$
$$\lambda_5= 0.1650139855-0.3234822044i,\ \ \lambda_6=0.1650139855+0.3234822044i.$$
 Hence $|\lambda_1|<1$, $|\lambda_2|<1$, $|\lambda_5|=|\lambda_6|=0.131870352<1$, but
 $|\lambda_3|>1, \ \ |\lambda_4|>1$.  Therefore for the baseline parameters
 the fixed point $P_1$ is a saddle point.

 From the known theorem about stable and unstable manifolds (see \cite{De} and  \cite{G})
 we get the following result
 \begin{pro} If parameters of operator $M$ given by (\ref{dso}) are as in (\ref{base}) then \begin{itemize}
 \item[a.] There is a two-dimensional (resp. four-dimensional) smooth manifold,
 denoted by $W^{s}(P_0)$ (resp. $W^{s}(P_1)$) such that for any initial
 vector $v\in {\displaystyle W^{s}(P_i)}$ one has
 $$\lim_{n\to\infty} M^n(v)=P_i, \, i=0,1.$$
 \item[b.] There is a four-dimensional (resp. two-dimensional) smooth manifold, denoted by  $W^{u}(P_0)$
 (resp. $W^{u}(P_1)$) and a neighborhood $V(P_i)$ of $P_i$ such that for any initial vector $v\in W^{u}(P_i)\cap V(P_i)$, there exists $k=k(v)\in N$ that $M^k(v)\notin V(P_i)$, $i=0,1$.
 \end{itemize}
  \end{pro}
The set $W^{s}(P_i)$ is known as a stable manifold and $W^{u}(P_i)$ is an unstable manifold.

Using linearity of five coordinates of the operator (\ref{dso}) one can solve
the equation $M(M(v))=v$, which gives 2-periodic points of the dynamical system.
Since there are 15 parameters the solutions have very long formulas, but computer analysis
will be helpful to see that for some parameters there are two-periodic points
different from the fixed points.

\section{Evolution algebras of mosquito population.}

\subsection{Definitions.}

Let $(\mathcal{E},\cdot )$ be an algebra over a field $\mathbb{K}$ ($\mathbb{%
R}$ or $\mathbb{C}).$ If $\ \mathcal{E}$ admits a basis $%
\{e_{1},e_{2},\dots \}$, such that
\begin{equation*}
e_{i}\cdot e_{j}=%
\begin{cases}
0, & \text{if \ $i\neq j$;} \\
\displaystyle\sum_{k}a_{ik}e_{k}, & \text{if \ $i=j$,}%
\end{cases}%
\end{equation*}%
then this algebra is called an \textit{evolution algebra} \cite{t}. The
basis is called a natural basis. We denote by ${\mathbf A}=(a_{ij})$ the matrix of the
structural constants of the evolution algebra $\mathcal{E}$.

It is known that an evolution algebra is commutative but not associative, in
general. For basic properties of the evolution algebra see \cite{t} and
\cite{ca-si-ve}.

For an evolution algebra $\mathcal{E}$ and $k\geq 1$ we introduce the
following sequence
\begin{equation}
\mathcal{E}^{k}=\sum_{i=1}^{k-1}\mathcal{E}^{i}\mathcal{E}^{k-i},  \label{EE}
\end{equation}%
where $\mathcal{E}^{2}:=\{a\cdot b:a,b\in \mathcal{E}\}.$

Since $\mathcal{E}$ is a commutative algebra we obtain
\begin{equation*}
\mathcal{E}^{k}=\sum_{i=1}^{\lfloor k/2\rfloor }\mathcal{E}^{i}\mathcal{E}%
^{k-i},
\end{equation*}%
where $\lfloor x\rfloor $ denotes the integer part of $x$.

\begin{defn}
An evolution algebra $\mathcal{E}$ is called nilpotent if there exists some $%
n\in \mathbb{N}$ such that $\mathcal{E}^n=0$. The smallest $n$ such that $%
\mathcal{E}^n=0$ is called the index of nilpotency.
\end{defn}

The following theorem is known (see \cite{CLOR}).

\begin{thm}
\label{t2} An $n$-dimensional evolution algebra $\mathcal{E}$ is nilpotent
iff the matrix of the structural constants corresponding to $\mathcal{E}$
can be written as
\begin{equation}  \label{5}
\widehat{A}=\left(
\begin{array}{ccccc}
0 & a_{12} & a_{13} & \dots & a_{1n} \\[1.5mm]
0 & 0 & a_{23} & \dots & a_{2n} \\[1.5mm]
0 & 0 & 0 & \dots & a_{3n} \\[1.5mm]
\vdots & \vdots & \vdots & \cdots & \vdots \\[1.5mm]
0 & 0 & 0 & \cdots & 0 \\
&  &  &  &
\end{array}
\right).
\end{equation}
\end{thm}

Now we give an algebra structure on the vector space $\mathbb{R}^{6}$ which
is closely related to the map defined by (\ref{dso}). Let $\mathbb{E}%
_{\epsilon }\equiv \mathbb{E}_{\epsilon ,M}$ be a 6-dimensional evolution
algebra over the set of real numbers, with the natural basis $%
\{e_{1},...,e_{6}\}$ and multiplication table $e_{i}e_{j}=0$ if $i\neq j$,
\begin{equation}
\begin{array}{llllll}
e_{1}^{2}=(1-\hat{e})e_{1}+b\theta e_{6}, &  &  &  &  &  \\[3mm]
e_{2}^{2}=ee_{1}+(1-\hat{l}_{1}-2\hat{l}_{2}\epsilon )e_{2}, &  &  &  &  &
\\[3mm]
e_{3}^{2}=ae_{2}+(1-\hat{p})e_{3}, &  &  &  &  &  \\[2mm]
e_{4}^{2}=pe_{3}+(1-\hat{h})e_{4}+\theta e_{6}, &  &  &  &  &  \\[2mm]
e_{5}^{2}=he_{4}+(1-\hat{r})e_{5}, &  &  &  &  &  \\[2mm]
e_{6}^{2}=re_{5}+(1-\hat{\theta})e_{6}, &  &  &  &  &
\end{array}
\label{dsa}
\end{equation}%
where all parameters coincide with parameters of the operator (\ref{dso})
and $\epsilon \in \{0,L^{\ast }\}$, with $L^{\ast }$ is defined in
Proposition \ref{p1}. Thus the matrix of structural constants ${\mathbf A}_{\epsilon
}=(a_{ij})$ of this algebra $\mathbb{E}_{\epsilon }$ coincides with Jacobi
matrix defined in (\ref{JM}) calculated at fixed points (at $P_{0}$ for $%
\epsilon =0$, at $P_{1}$ for $\epsilon =L^{\ast }$). Namely, we consider two
Evolution algebras $\mathbb{E}_{\epsilon }$, where $\epsilon
=0,L^{\ast }$ respectively, corresponding to matrices ${\mathbf A}_{\epsilon
}$ of structural constants given by

\begin{equation}
{\mathbf A}_{0}=J_{M}(P_{0}),\ \ {\mathbf A}_{L^{\ast }}=J_{M}(P_{1}),  \label{A}
\end{equation}%
where $P_{i}$ is fixed point given in Proposition \ref{p1}.

By formula (\ref{dsa}), for any two vectors $x=(x_1,\dots,x_6), \,
y=(y_1,\dots,y_6)\in \mathbb{E}_\epsilon$ we get the following
multiplication
\begin{equation*}
xy=\Bigg((1-\hat e)x_1y_1+ex_2y_2,\, (1-\hat l_1-2\hat
l_2\epsilon)x_2y_2+ax_3y_3,\, (1-\hat p)x_3y_3+px_4y_4,
\end{equation*}
\begin{equation}  \label{xy}
(1-\hat h)x_4y_4+hx_5y_5,\, (1-\hat r)x_5y_5+rx_6y_6,\, (1-\hat
\theta)x_6y_6+b\theta x_1y_1+\theta x_4y_4\Bigg).
\end{equation}

\subsection{Idempotent and absolute nilpotent elements}
An element $x\in \mathbb E_\epsilon$ is called idempotent (resp. absolute nilpotent) if $x^2 = x$ (resp. $x^2=0$).
 Such points of $\mathbb E_\epsilon$ are especially important,
because they are the fixed points of the evolution map $V(x) = x^2$, where $x^2$, by (\ref{xy}) has the form
$$
x^2=\Bigg((1-\hat e)x_1^2+ex_2^2,\,  (1-\hat l_1-2\hat l_2\epsilon)x_2^2+ax_3^2,\, (1-\hat p)x_3^2+px_4^2,$$
\begin{equation}\label{xx}
(1-\hat h)x_4^2+hx_5^2,\, (1-\hat r)x_5^2+rx_6^2,\, (1-\hat \theta)x_6^2+b\theta x_1^2+\theta x_4^2\Bigg).
\end{equation}

\begin{pro}\label{pan} The algebra $\mathbb E_\epsilon$ has unique absolute nilpotent element $x=(0,...,0)$.
\end{pro}
\begin{proof}
We have
\begin{equation}\label{det}
\det({\mathbf A}_\epsilon)=(1-\hat e)(1-\hat l_1-2\hat l_2\epsilon)(1-\hat p)[(1-\hat h)(1-\hat r)(1-\hat\theta)+hr\theta]-abehpr\theta.
\end{equation}

It is easy to see that if $\det({\mathbf A}_\epsilon)\ne 0$ then the absolute nilpotent
element is unique for the algebra $\mathbb E_\epsilon$, i.e.
$x^2=0$ has unique solution $x=(0,\dots,0)$. In case $\det({\mathbf A}_\epsilon)=0$ we use Table 1 (see also (\ref{sim})),
from which we have that
$h>0$ and $\hat h<1$, therefore from the 4-th equation of $x^2=0$, i.e., $(1-\hat h)x_4^2+hx_5^2=0$
we get $x_4=x_5=0$, then consequently, by the direction  $5\to 6\to 1\to 2\to 3$ we get $x_6=x_1=x_2=x_3=0$.
\end{proof}

\begin{lemma}
\label{ld} Let ${\mathbf A}_{\epsilon }$ be the structure matrix (given in (2.5)) of the
evolution algebra $\mathbb{E}_{\epsilon }$

\begin{itemize}
\item If $(1-\hat{e})(1-\hat{l}_{1}-2\hat{l}_{2}\epsilon
)(1-\hat{p})=0$ then $\det ({\mathbf A}_{\epsilon })\neq 0$.

\item  $\mathrm{rank}({\mathbf A}_{\epsilon })\in \{5,6\}$.
\end{itemize}
\end{lemma}

\begin{proof}
The first part follows from the formula (\ref{det}) knowing that $abehpr\theta \neq 0$. For second part, since%
\begin{equation*}
{\mathbf A}_{\epsilon }=\left(
\begin{array}{cccccc}
1-\hat{e} & 0 & 0 & 0 & 0 & b\theta  \\[3mm]
e & 1-\hat{l}_{1}-2\hat{l}_{2}L & 0 & 0 & 0 & 0 \\[3mm]
0 & a & 1-\hat{p} & 0 & 0 & 0 \\[3mm]
0 & 0 & p & 1-\hat{h} & 0 & \theta  \\[3mm]
0 & 0 & 0 & h & 1-\hat{r} & 0 \\[3mm]
0 & 0 & 0 & 0 & r & 1-\hat{\theta}%
\end{array}%
\right)
\end{equation*}%
and the parameters  $e,a,o,h,r$ are non-zero it follows that the
corresponding columns are linearly independents and hence  $\mathrm{rank}%
({\mathbf A}_{\epsilon })\geq 5.$This completes the proof.
\end{proof}

For the idempotent elements we should solve the following system
\begin{equation}\label{ie}\begin{array}{llllll}
x_1=(1-\hat e)x_1^2+ex_2^2,\\[2mm]
x_2=(1-\hat l_1-2\hat l_2\epsilon)x_2^2+ax_3^2,\\[2mm]
x_3=(1-\hat p)x_3^2+px_4^2,\\[2mm]
x_4=(1-\hat h)x_4^2+hx_5^2,\\[2mm]
x_5=(1-\hat r)x_5^2+rx_6^2,\\[2mm]
x_6=(1-\hat \theta)x_6^2+b\theta x_1^2+\theta x_4^2.
\end{array}
\end{equation}
By Table 1, and (\ref{sim}) we have that
$$1-\hat h>0, \ \ 1-\hat r>0, \ \ 1-\hat \theta<0.$$
\begin{lemma}
\label{l1} For the system given in  (\ref{ie}), the following assertions
hold

\begin{itemize}
\item[1.] If $x_i=0$ for some $i=1,2,\dots,6$ in (\ref{ie}) then $%
x_1=\dots=x_6=0$.

\item[2.] Let $1-\hat e>0$, $1-\hat p>0$, $1-\hat l_1-2\hat l^2\epsilon>0$.
If $(x_1, \dots, x_6)$ is a solution to (\ref{ie}) then
\begin{equation*}
x_1\in (0,{\frac{1}{1-\hat e}}], \ \ x_2\in (0,{\frac{1}{1-\hat l_1-2\hat
l_2\epsilon}}], \ \ x_3\in (0,{\frac{1}{1-\hat p}}],
\end{equation*}
\begin{equation*}
x_4\in (0,{\frac{1}{1-\hat h}}], \ \ x_5\in (0,{\frac{1}{1-\hat r}}], \ \
x_6\in (-\infty, {\frac{1}{1-\hat \theta}}]\cup (0,+\infty).
\end{equation*}

\item[3.] If $1-\hat e<0$ or $1-\hat p<0$ or $1-\hat l_1-2\hat l^2\epsilon<
0 $ then the corresponding $x_i$ may be positive and negative (as $x_6$).
\end{itemize}
\end{lemma}
\begin{proof} 1. Straightforward.

2. From the first equation of (\ref{ie}) we have
$$x_1-(1-\hat e)x_1^2=ex_2^2\geq 0, \ \ \mbox{i.e.} \ \ x_1\in [0,{1\over 1-\hat e}].$$
Other parts of the Lemma can be obtained similarly, taking into account the sign of the parameters.
\end{proof}
For $t$ such that $t-at^2\geq 0$ we introduce the following function
$$f_{a,b}(t)=\left({t-at^2\over b}\right)^{1/2}, \ \ a\in \mathbb R,\ \ b>0.$$
The following lemma reduces the nonlinear system (\ref{ie}) of six unknown to
equations with only one unknown.
\begin{lemma}\label{l2} If $(x_1,\dots, x_6)$ is a solution to (\ref{ie}) then $x_1$
satisfies one of the following equations:
\begin{equation}\label{x1}
\pm g(x_1)=(1-\hat\theta)(g(x_1))^2+b\theta x_1^2+\theta (\varphi(x_1))^2,\end{equation}
 where
$$g(x)=f_{1-\hat r, r}(f_{1-\hat h, h}(\varphi(x))),$$ and\footnote{Depending on sign of parameters $1-\hat p$, $1-\hat l_1-2\hat l^2\epsilon$ the function $\varphi$ has 4 forms. For example, in case of part 2) of Lemma \ref{l1} the function is  $\varphi(x)=f_{1-\hat p, p}(f_{1-\hat l_1-2\hat l_2\epsilon, a}(f_{1-\hat e, e}(x)))$, and if $1-\hat p<0$, $1-\hat l_1-2\hat l^2\epsilon>0$  then  $\varphi(x)=f_{1-\hat p, p}(-f_{1-\hat l_1-2\hat l_2\epsilon, a}(f_{1-\hat e, e}(x)))$.}
$$\varphi(x)=f_{1-\hat p, p}(\pm f_{1-\hat l_1-2\hat l_2\epsilon, a}(\pm f_{1-\hat e, e}(x))).$$
\end{lemma}
\begin{proof} From the first equation of (\ref{ie})
we have
\begin{equation}\label{2x}
x_2=\pm f_{1-\hat e, e}(x_1)
\end{equation}
and from the second equation we have
\begin{equation}\label{3x}
x_3=\pm f_{1-\hat l_1-2\hat l_2\epsilon, a}(x_2).
\end{equation}
Note that the sign $\pm$ in (\ref{2x}) depends on the values of $1-\hat l_1-2\hat l_2 \epsilon$:
if the last number is positive then $x_2\in [0, {1\over  1-\hat l_1-2\hat l_2 \epsilon}]$; if
$1-\hat l_1-2\hat l_2 \epsilon<0$ then $x_2\in (-\infty, {1\over  1-\hat l_1-2\hat l_2 \epsilon}]\cup [0,+\infty)$;
but if  $1-\hat l_1-2\hat l_2 \epsilon=0$ then $x_2\geq 0$.
The sign $\pm$ in (\ref{3x}) depends on value of $1-\hat p$ similarly as in (\ref{2x}).

From the third equation we get
\begin{equation}\label{4x}
x_4=f_{1-\hat p, p}(x_3).
\end{equation}
In this equation do not consider the case `-' because $1-\hat h>0$ and value of
$x_4$ by the fourth equation of (\ref{ie}) can be $x_4\in [0, {1\over 1-\hat h}]$.
Similarly, from the fourth equation for $x_5$ we have (positive value):
\begin{equation}\label{5x}
x_5=f_{1-\hat h, h}(x_4).
\end{equation}
Consequently, for $x_6$ we have (because $1-\hat\theta<0$):
\begin{equation}\label{6x}
x_6=\pm f_{1-\hat r, r}(x_5).
\end{equation}
Now using (\ref{2x})-(\ref{6x}), we can write each $x_i$, $i=2,\dots,6$ as a function of $x_1$.
After that, from the sixth equation of (\ref{ie}) we get (\ref{x1}).
\end{proof}
From proof of Lemma \ref{l2} we get
\begin{cor} If $(x_1,\dots,x_6)$ is a solution to (\ref{ie}) then $x_4\geq 0$, $x_5\geq 0$.
\end{cor}
For each given concrete values of parameters, one can solve equation (\ref{x1}) using a computer (Maple or Mathematica).

Here we illustrate this for the case of parameters (\ref{base}) and algebra $\mathbb E_0$, i.e. $\epsilon=0$.
In this case we have $1-\hat e=-0.06<0$, $1-\hat p=0.13>0$, and $1-\hat l_1=0.42>0$. We should have
a solution, $x_1$, to (\ref{x1}), such that $x_1\in (-\infty, -{1\over 0.06})\cup (0,+\infty)$.
This $x_1$ should satisfy one of the following two equations (the parameters as in (\ref{base}))
  \begin{equation}\label{x1b}
g(x_1)=-2.41(g(x_1))^2+300x_1^2+3(\varphi(x_1))^2,\end{equation}
  \begin{equation}\label{x1c}
-g(x_1)=-2.41(g(x_1))^2+300x_1^2+3(\varphi(x_1))^2,\end{equation}
 where
$$g(x)=f_{0.5657, 0.43}(f_{0.36, 0.46}(\varphi(x))),$$
$$\varphi(x)=f_{0.13, 0.5}(f_{0.42, 0.14}(f_{-0.06, 0.5}(x))).$$
{\it Case:} (\ref{x1b}). As Maple analysis shows the equation (\ref{x1b}) has a solution $x_1$,
which by (\ref{2x})-(\ref{6x}) generates all $x_i$:
\begin{equation}\label{i1}
\begin{array}{ll}
x_1=0.0003369672, \ \ x_2=0.02596050896, \ \ x_3=0.4282643609,\\
x_4=0.8993564519, \ \ x_5=1.149832995, \ \ \ \ x_6=0.966788043.
\end{array}
\end{equation}
{\it Case:} (\ref{x1c}). The equation (\ref{x1c}) has a solution $x_1$,
which by (\ref{2x})-(\ref{6x}) generates all $x_i$:
\begin{equation}\label{i2}
\begin{array}{ll}
x_1=0.000260712, \ \ x_2=0.02283488902, \ \ x_3=0.4019229449,\\
x_4=0.8728373021, \ \ x_5=1.140721661, \ \ x_6=-0.9700237961.
\end{array}
\end{equation}
Summarizing we have
\begin{pro}\label{pie} The evolution algebra $\mathbb E_\epsilon$ corresponding to baseline parameters (\ref{base}) has at least three
idempotent elements: $(0,\dots,0)$ and the elements given by the coordinates (\ref{i1}) and (\ref{i2}).
\end{pro}

It is easy to see the following equalities (see \cite[page 27]{t}):
\begin{equation*}
e_{i}^{m}=a_{ii}^{m-2}e_{i}^{2},\ \ m\geq 2.
\end{equation*}%
By Table 1 we have that $a_{ii}\in (0,1)$ for $i=2,3,4,5$ and $a_{ii}>1$ for
$i=1,6$. Therefore we have
\begin{equation*}
\lim_{m\rightarrow \infty }e_{i}^{m}=\left\{
\begin{array}{ll}
0,\ \ \mbox{if}\ \ i=2,3,4,5 &  \\[2mm]
\infty ,\ \ \mbox{if}\ \ i=1,6. &
\end{array}%
\right.
\end{equation*}

Comparing with Theorem \ref{t2} one can see that the algebra $\mathbb{E_\epsilon}$ is
not nilpotent, because take for example $e_{i}^{n}$ which is not zero for
any $n\geq 1$.

\subsection{Simplicity.}

A subalgebra of an algebra is a subset of elements that is closed under addition,
multiplication, and scalar multiplication.
An ideal of a commutative algebra is a linear subspace that has the property that
any element of the subspace multiplied by any element of the algebra
produces an element of the subspace.

We recall that an algebra is simple whenever it has non-zero product and
has no non-zero proper ideals.

\begin{defn}
As defined in \cite[Definitons 3.1]{ca-si-ve}, let \textrm{\ }$B=\{e_{i}\
:i\in \Lambda \}$ be a natural basis of an evolution algebra $E$\ and let $%
i_{0}\in \Lambda .$\ The first-generation descendents of $\ i_{0}$ \ are the
elements of the subset $D^{1}(i_{0})$\ given by:
\begin{equation*}
D^{1}(i_{0}):=\left\{ k\in \Lambda \ |\ e_{i_{0}}^{2}=
\sum_{k}a_{i_{0}k}e_{k}\text{ with }a_{i_{0}k}\neq 0\right\} .
\end{equation*}%
In an abbreviated form, $D^{1}(i_{0}):=\{j\in \Lambda \ |\ a_{i_{0}j}\neq 0\}.$

Similarly, we say that $j$\ is a \textit{second-generation descendent} of $%
i_{0}$\ whenever $j\in D^{1}(k)$\ for some $k\in D^{1}(i_{0}).$\ Therefore,
\begin{equation*}
D^{2}(i_{0})=\bigcup\limits_{k\in D^{1}(i_{0})}D^{1}(k).
\end{equation*}%
By recurrence, we define the set of \textit{m-th-generation descendants}  of $%
i_{0}$\ as
\begin{equation*}
{\ }D^{m}(i_{0})=\bigcup\limits_{k\in D^{m-1}(i_{0})}D^{1}(k).
\end{equation*}%
\ Finally, the \ set of \textit{descendants}of $i_{0}$\ is defined as the
subset of $\Lambda $\ given by
\begin{equation*}
D(i_{0})=\bigcup\limits_{m\in \mathbb{N}}D^{m}(i_{0}).
\end{equation*}
\end{defn}

Simple finite-dimensional evolution algebras were characterized in \cite[%
Corollary 4.10]{ca-si-ve} as those evolution algebras $E$ with a natural
basis \thinspace $B=\{e_{i}:i\in \Lambda \}$ such that the determinant of
the  corresponding matrix of structural constants is non-zero and $\Lambda =D(i)$
for every $i\in \Lambda$.

\begin{thm}\label{ts}
The algebra $\mathbb E_\epsilon$ is simple if and only if $\det(\mathbf A_\epsilon)\ne 0$.
\end{thm}

\begin{proof}
We use the above-mentioned \cite[Corollary 4.10]{ca-si-ve}. It is easy to see that
the graph associated with matrix ${\mathbf A}_\epsilon$ is cyclic, in the sense that given two
vertices there is always a path from one to the other one (see \cite[Remark 4.11]{ca-si-ve}).
Consequently, for our algebra of mosquito the condition ``$\Lambda =D(i)$
for every $i\in \Lambda$" is always satisfied. Thus from condition  $\det(\mathbf A_\epsilon)\ne 0$
it follows that the algebra is simple.

Conversely, assume that  $\mathbb E_\epsilon$ is simple, then from  \cite[Corollary 4.10]{ca-si-ve} it follows that $\det(\mathbf A_\epsilon)\ne 0$.
\end{proof}
Next we determine how are the non-zero proper ideals of $\mathbb E_\epsilon$
whenever this algebra is not simple.

According with Table 1 and equations (\ref{ds}),
(\ref{sim}) we recall ranges of parameters:
\begin{equation}\label{pa}\begin{array}{ccc}
 1-\hat e\in [-0.8, 0.33],& \theta\in [3, 4], & 1-\hat l_1-2\hat l_2 \epsilon \in [0.25, 0.62-2\epsilon],\\[2mm]
 e\in [0.33, 1], & 1-\hat r\in [0.43, 0.697], & 1-\hat p\in [-0.52, 0.45],\\[2mm]
 l\in [0.08, 0.17],& 1-\hat \theta\in [-3.56, -2.41], & 1-\hat h\in [0.169, 0.553],\\[2mm]
 p\in [0.33,1], & b\theta\in [150,1200],& h\in [0.322, 0.598],\\[2mm]
 r\in [0.3, 0.56],&a\in [0.08, 0.17]. &
 \end{array}
\end{equation}

\begin{thm}\label{tns}
Let $I$ be a non-zero proper ideal of $\mathbb E_\epsilon$ (i.e. the case $%
\det ({\mathbf{A}}_{\epsilon })=0$). Then
\begin{itemize}
\item[a)] $\dim I=5$ and $I=lin\{e_{1}^{2},...,e_{5}^{2}\}=lin\{e_{1}^{2},...,e_{6}^{2}\}.$
\item[b)]
  $e_{k}\notin I$ for every $i=1,2,...,6.$
  \end{itemize}
\end{thm}

\begin{proof}
a) Since $\det ({\mathbf{A}}_{\epsilon })=0$ (otherwise $\mathbb E_\epsilon$ is simple by
the above theorem)\ it is easy to check that $e_{6}^{2}$ is a linear
combination of the linearly independent vectors $%
e_{1}^{2},e_{2}^{2},e_{3}^{2},e_{4}^{2},e_{5}^{2}$ (see Lemma 1 in Section
3.2) so that
\begin{equation}
\
lin\{e_{1}^{2},e_{2}^{2},e_{3}^{2},e_{4}^{2},e_{5}^{2}\}=lin%
\{e_{1}^{2},e_{2}^{2},e_{3}^{2},e_{4}^{2},e_{5}^{2},e_{6}^{2}\}.  \label{equ}
\end{equation}

Let $I$ $\ $be a non-zero proper ideal of $\mathbb E_\epsilon$ and let
\begin{equation*}
k_{0}:=\min \{k\in \{1,2,3,4,5,6\}:\pi _{k}(I)\neq 0\},
\end{equation*}%
where $\pi _{k}$ denotes the projection over $\mathbb{K}e_{k}.$
For this $k_0$ there exists $x\in I$ and $c\ne 0$ such that
$x=ce_{k_0}+\dots$. Multiplying $x$ to $e_{k_0}$ we get $xe_{k_0}=ce_{k_0}^2\in I$
therefore $e_{k_{0}}^{2}\in I$.
Now using (\ref{dsa}) we get that $e_{k}^{2}\in I$ $\ $%
for {\it any} $k=1,2,...,6,$ so that
\begin{equation*}
I=\
lin\{e_{1}^{2},e_{2}^{2},e_{3}^{2},e_{4}^{2},e_{5}^{2}\}=lin%
\{e_{1}^{2},e_{2}^{2},e_{3}^{2},e_{4}^{2},e_{5}^{2},e_{6}^{2}\}\text{ }
\end{equation*}%
and $\dim I=5.$

b) To prove part b) we show that if $e_k\in I$ for some $k=1,2,\dots, 6$
then $e_i\in I$ for any $i=1,\dots,6$, i.e., $I=\mathbb E_\epsilon$,  which is non-proper.
Start from $e_1\in I$:
\begin{itemize}
\item[1)] if $e_1\in I$ then by (\ref{dsa}) and (\ref{pa})
there are $c_6$, $d_6$ such that $e_6=c_6e_1^2-d_6e_1\in I$;
\item[2)] Now from $e_{6}\in I$ by the last equality of (\ref{dsa})
 we see that there are $c_5$ and $d_5$ such that
$e_{5}=c_5e_{6}^{2}-d_5e_{6}\in I$;
\item[3)] From $e_{5}\in I$ by (\ref{dsa}) it follows
that there are $c_4$ and $d_4$ such that
$e_{4}=c_4e_{5}^{2}-d_4e_{5}\in I$;
\item[4)] Using 1)-3), i.e., $e_{4}^2, e_4, e_6\in I$, by (\ref{dsa})
 there are $c_3$, $d_3$ and $\tilde d_3$ such that
$e_{3}=c_3e_{4}^{2}-d_3e_{4}-\tilde d_3e_6\in I$;
\item[5)] Since $e_{3}\in I$ we have $c_2$ and $d_2$ such that
$e_{2}=c_2e_{3}^{2}-d_2e_{3}\in I$;
\item[6)] Since $e_{2}\in I$ we have $c_1$ and $d_1$ such that
$e_{1}=c_1e_{2}^{2}-d_1e_{2}\in I$.
\end{itemize}
Thus we have the following cycle:
\begin{equation}\label{cy1}
e_1\in I\Rightarrow e_6\in I\Rightarrow e_5\in I\Rightarrow
e_4\in I\Rightarrow e_3\in I\Rightarrow e_2\in I\Rightarrow e_1\in I.\end{equation}
Note that in the cycle (\ref{cy1}) if we start from any $e_k\in I$, $k=1,2,3,5,6$ (except $k=4$) then
the cycle runs recurrently depending only on previous 'state'.
But if we start from $k=4$ then as part 4) above shows
that to run our cycle we need $e_6\in I$. To avoid this we construct the following new cycle:
\begin{itemize}
\item[i)] start from $e_4\in I$ then, since $e_5^2\in I$,  by (\ref{dsa}) and (\ref{pa})
there are $\alpha_5$, $\beta_5$ such that $e_5=\alpha_5e_5^2-\beta_5e_4\in I$;
\item[ii)] using $e_{5}\in I$, $e_6^2\in I$ by the last equality of (\ref{dsa})
 we see that there are $\alpha_6$ and $\beta_6$ such that
$e_{6}=\alpha_6e_{6}^{2}-\beta_6e_{5}\in I$;
\item[iii)] using $e_4, e_4^2, e_6\in I$ from 4-th equation of (\ref{dsa})
we get $e_3=\alpha_3e_4^2-\beta_3 e_4-\gamma_3 e_6\in I$.
\item[iv)] using $e_{3}\in I$, $e_3^2\in I$ by the third equation of (\ref{dsa}) it follows
that there are $\alpha_2$ and $\beta_2$ such that
$e_{2}=\alpha_2e_{3}^{2}-\beta_2e_{3}\in I$.
\item[iv)] Finally, from the second equation of (\ref{dsa}) we get $e_1=\alpha_1e_2^2-\beta_1e_2\in I$.
 Thus we constructed the following cycle
\begin{equation}\label{cy2}
e_4\in I\Rightarrow e_5\in I\Rightarrow e_6\in I\Rightarrow
e_3\in I\Rightarrow e_2\in I\Rightarrow e_1\in I\Rightarrow  {\rm continued \, by} \, (\ref{cy1}).\end{equation}
\end{itemize}
These two cycles show that as soon as $e_k\in I$ for {\it some} $k=1,\dots 6$ then $e_i\in I$ for {\it all} $i=1,\dots,6$.
Then it follows that $I=$ $\mathbb E_\epsilon$, i.e., $I$ is not proper.
\end{proof}
From the proof of the theorem we get
\begin{cor} Let $I$ be a non-zero ideal of $\mathbb E_\epsilon$ then it is proper if and only if
 $e_{k}\notin I$ for every $i=1,2,...,6.$
\end{cor}

The above result shows that $\mathbb E_\epsilon$ is not decomposable as a direct sum
of ideals.

The Jacobson radical of a commutative algebra without a unit (as $\mathbb E_\epsilon$)
is defined as the intersection of all its maximal modular ideals. \ We
note that if the algebra has a unit then maximal modular ideals are nothing
but maximal ideals. But this is not the case $\mathbb E_\epsilon$ because as showed in
\cite[page 22]{t} if an evolution algebra has a unit then it is a non-zero
trivial algebra.

We recall that an algebra is said to be semisimple if its Jacobson radical
is zero and it said to be a radical algebra whenever the algebra has no
ideals of this type.

In \cite[Corollary 3.12]{Ve} maximal modular ideals of an evolution algebra
were characterized in it terms of modular indexes (those index $i$ such that $\pi_i(e_i^2)\ne 0$ and 
$\pi_i(e_k^2)=0$ if $k\ne i$). Since the structure matrix ${\bf A}_\epsilon$ has not a modular index,
unless $1-\hat e\neq 0$, the following result follows:

\begin{pro}
The algebra $\mathbb E_\epsilon$ is a radical algebra.
\end{pro}

The above proposition shows that $\mathbb E_\epsilon$ has a 'bad' behavior. In
fact the algebras that enjoy nice properties (as the automatic continuity of
every surjective homomorphims from a Banach algebra onto them) are the
semisimple ones (that is those whose Jacobson radical is zero).
Thus $\mathbb E_\epsilon$ is a chaotic algebra.

By (\ref{ee}) it follows that $\lambda =1$ is an eigenvector of ${\bf A}_\epsilon$ iff
\begin{equation}\label{ee1}
\hat e\hat p(\hat l_1+2\hat l_2\epsilon)[\hat h\hat r\hat\theta-hr\theta]=abehpr\theta.\end{equation}

In the next section,  depending on the fact that $\lambda =1$ be or not an
eigenvalue of ${\bf A}_\epsilon$ we will describe limit points of the
dynamical system (generated by a linear operator) as well as initial points to reach them.

\subsection{A subset of limit points of the evolution operator}
Following \cite{t} we define an evolution operator as a linear map $\mathcal L\equiv \mathcal L_\epsilon$ to be
$\mathcal L: \mathbb E_\epsilon\to \mathbb E_\epsilon$ as $\mathcal L(x)={\bf A}_\epsilon x$, i.e.,
${\bf A}_\epsilon$ is the matrix of the linear map. This also can be written as $\mathcal L(x)=\ell x$, with $\ell=e_1+e_2+\dots+e_6$.

Denote by $\sigma({\bf A}_\epsilon)$ the set of all eigenvalues $\lambda_i$, $i=1,\dots,6$ (spectrum) of
${\bf A}_\epsilon$.

Let $\mathcal L^n$ be $n$-th iteration of $\mathcal L$.

\begin{pro}\label{pl} Let $\lambda\in \sigma({\bf A}_\epsilon)$ and $1\notin \sigma({\bf A}_\epsilon)$. For an eigenvector  $c$ corresponding to $\lambda$ (i.e. $\mathcal L(c)=\lambda c$)
define $b_c=(\mathcal L-I)^{-1}c$
then $\mathcal L$ has unique fixed point 0 and
$$\lim_{n\to\infty}\mathcal L^{n}(b_c)=\left\{\begin{array}{lll}
0, \ \  \mbox{if} \ \ |\lambda|<1\\[2mm]
\infty, \ \ \mbox{if} \ \ |\lambda|>1,\\[2mm]
b_c, \ \ \mbox{if} \ \ n=2k, \, \lambda=-1,\\[2mm]
b_c+c, \ \ \mbox{if} \ \ n=2k+1, \, \lambda=-1.\\[2mm]
\end{array}\right.$$
\end{pro}
\begin{proof}
Using $\mathcal L(x)=\ell x$ and $\ell c=\lambda c$ we get
$$c=(\mathcal L-I)b_c  \Rightarrow \ell b_c=b_c+c.$$
$$\mathcal L(b_c)=\ell b_c=b_c+c,$$
$$\mathcal L^2(b_c)=\ell(\ell b_c)=\ell (b_c+c)=\ell b_c+\ell c=b_c+c+\lambda c=b_c+(\lambda+1)c.$$
$$\mathcal L^3(b_c)=\ell(b_c+(\lambda+1)c)=\ell b_c+(\lambda+1)\ell c= b_c+c+(\lambda+1)\lambda c=b_c+(\lambda^2+\lambda+1)c.$$
Using induction over $n$ we show that
$$L^n(b_c)=b_c+(\sum_{i=0}^{n-1}\lambda^i)c.$$
For $n=1,2,3$ we have already checked this formula, now assuming that it is true for $n$, we show it for $n+1$:
 $$L^{n+1}(b_c)=\ell(b_c+(\sum_{i=0}^{n-1}\lambda^i)c)=b_c+c+(\sum_{i=0}^{n-1}\lambda^i)\lambda c=b_c+(\sum_{i=0}^{n}\lambda^i)c.$$

{\it Case:} $|\lambda|<1$. Then
$$\lim_{n\to\infty}\sum_{i=0}^{n-1}\lambda^i={1\over 1-\lambda}.$$
$$\lim_{n\to\infty}\mathcal L^{n}(b_c)=b_c+{c\over 1-\lambda}.$$
We show that $(\mathcal L-I)^{-1}c+{c\over \lambda}$ is a fixed point:
$$\mathcal L\left((\mathcal L-I)^{-1}c+{c\over 1-\lambda}\right)=\ell(b_c+{c\over 1-\lambda})
=\ell b_c+{\ell c\over 1-\lambda}$$ $$=
b_c+c+{\lambda c\over 1-\lambda}=b_c+{c\over 1-\lambda}= (\mathcal L-I)^{-1}c+{c\over 1-\lambda}.$$
Now we prove that if $1\notin \sigma({\bf A}_\epsilon)$ then only 0 is fixed point of $\mathcal L$:
assume it has some non-zero fixed point, say $x^*$, i.e. $\mathcal L(x^*)=x^*$. Then any element
of the form $x=\alpha x^*$ satisfies the equation $\mathcal L(x)=x$, this is a contradiction
to the condition that $1\notin \sigma({\bf A}_\epsilon)$. Thus $(\mathcal L-I)^{-1}c+{c\over 1-\lambda}=0$.

{\it Case:} $|\lambda|>1$. In this case
$$\lim_{n\to\infty}\sum_{i=0}^{n-1}\lambda^i=
\lim_{n\to\infty}{\lambda^n-1\over \lambda-1}=
\left\{\begin{array}{ll}
+\infty, \ \ \mbox{if} \ \ \lambda>1,\\[2mm]
\pm \infty, \ \ \mbox{depending on parity of $n$, if} \ \ \lambda<-1
\end{array}\right.
$$

{\it Case:} $\lambda=-1$
$$\lim_{n\to\infty}\sum_{i=0}^{n-1}\lambda^i=
\lim_{n\to\infty}{(-1)^n-1\over -2}=
\left\{\begin{array}{ll}
0, \ \ \mbox{if} \ \ n=2k, \\[2mm]
1, \ \ \mbox{if} \ \ n=2k+1
\end{array}\right.
$$
and the proof is completed.
\end{proof}
Let us give some remarks.

{\bf Remark 2.} Concerning to Proposition \ref{pl} we note that
\begin{itemize}
\item[1.] By assumption $1\notin \sigma({\bf A}_\epsilon)$ we have $\lambda\ne 1$.
\item[2.] The infinite value of the limit means that $|\mathcal L^n|\to +\infty$ as $n\to \infty$.
\item[3.] The case $\lambda=-1$ means that $\mathcal L(c)=-c$, for any eigenvector $c$.
Then it follows that any such $c$ is a two-periodic point, i.e. $\mathcal L^2(c)=\mathcal L(-c)=-\mathcal L(c)=c$.
\item[4.] Form conditions of Proposition \ref{pl} remains the case $|\lambda|=1$ and $\lambda$ is a complex number.
In this case one can show that the limit does not exist, because the sequence will depend on
$\lambda^n=\cos(n\varphi)+i\sin(n\varphi)$. Depending on $\varphi\in [0, 2\pi]$ the set of limit points can be
a finite of an infinite set.
\item[5.] Note that for the baseline parameters (mentioned in the previous section),
from  Proposition \ref{pl} it follows that the corresponding operator has limit $0$ or $\infty$.
\end{itemize}

Denote
$$\mathcal M={\rm span}\{c\in {\rm Ker}({\bf A}_\epsilon-\lambda I): \lambda\in \sigma({\bf A}_\epsilon), \, |\lambda|<1\},$$
where span($S$) denotes the set of all finite linear combinations of elements of $S$.
\begin{pro}\label{pu} The following assertions hold
\begin{itemize}
\item[(i)] If $1\in  \sigma({\bf A}_\epsilon)$ then for any $v\in \mathcal M$ and any $b\in {\rm Ker}({\bf A}_\epsilon-I)$
(i.e. $b$ is a fixed point of $\mathcal L$)
the following holds
$$\lim_{n\to\infty}\mathcal L^{n}(v+b)=b.$$
\item[(ii)] $1\notin  \sigma({\bf A}_\epsilon)$ then for any $v\in \mathcal M$
the following holds
$$\lim_{n\to\infty}\mathcal L^{n}(v)=0.$$
\item[(iii)] Let $\rho({\bf A}_\epsilon)$ denote the spectral radius of ${\bf A}_\epsilon$, i.e.
the largest absolute value of its eigenvalues. Then $\lim_{n\to \infty }\mathcal L^n$ exists if and
only if $\rho({\bf A}_\epsilon) < 1$ or
$\rho({\bf A}_\epsilon) = 1$, where $\lambda = 1$ is the only eigenvalue on the
unit circle, and $\lambda = 1$ is semisimple.
Moreover, when the limit exists,
$$\lim_{n\to\infty}\mathcal L^n=\mbox{the projector onto the set of all eigenvectors associated
with}  \ \lambda=1 \ \ (eigenspace)$$ $$\mbox{along the range of a matrix} \ \ I - {\bf A}_\epsilon\, (imege space).$$
\end{itemize}
\end{pro}
\begin{proof} (i) Any vector $v\in \mathcal M$ has the form
$$v=\sum_i \alpha_i c_i, \ \ c_i\in {\rm Ker}({\bf A}_\epsilon-\lambda_i I), \, |\lambda_i|<1.$$
Consequently,
$$\mathcal L(v+b)=\mathcal L(v)+\mathcal L(b)=\sum_i \alpha_i \lambda_i c_i+b.$$
By induction one can see that
\begin{equation}\label{vb}
\mathcal L^n(v+b)=\sum_i \alpha_i \lambda^n_i c_i+b.
\end{equation}
Taking limit from the both sides of this equality completes the proof.

(ii) This follows from Proposition \ref{pl}.

(iii) This is a known theorem (see \cite[page 630]{Ma}).
\end{proof}

{\bf Remark 3.} Concerning to Proposition \ref{pu} we note that
\begin{itemize}
\item[1.] In case $1\in \sigma({\bf A}_\epsilon)$  (resp. $1\notin \sigma({\bf A}_\epsilon)$) we have $\det({\bf A}_\epsilon-I)=0$ (resp. $\ne 0$) therefore $({\bf A}_\epsilon-I)b=0$ has infinitely many (resp. unique)
solutions, having the form $\alpha b$ (resp. 0), which are fixed points of the operator $\mathcal L$.

\item[2.] In case when limit of $\mathcal L^n(x)$ does not exist, for some $x\in \mathbb E_\epsilon$, but the set of limit points is finite, then the sequence $\{\mathcal L^n(x)\}$ is asymptotically period, say with a period $p$. For such a sequence, one can use the part (iii) of Proposition \ref{pu} for the linear function $\mathcal L^p$, to investigate the limit $\lim_{k\to\infty}\mathcal L^{pk+i}(x)$,  $i=0,1,\dots,p-1$.
    \end{itemize}

\section{Biological interpretations}

A population biologist is interested in the long-term  behavior of the population
of a certain species or collection of species. Namely, the population biologist is interested in
what happens to an initial population of members. Does the population tend to zero as time goes on, leading to
extinction of the species? Does the population become arbitrarily large? Here we give some answers to these questions
related to the mosquito population.

Each point (vector) $v=(v_1,\dots,v_6)\in \mathbb R_+^6$ can be considered as a
state of the mosquito population, which is a
measure on the set $\{E, L, P, A_h, A_r, A_o\}$. If, for example, the value of $v_2$ is
close to zero, biologically this means that the contribution of the larva stage $L$ is small in future of the population.

The dynamical systems considered in this paper are interesting because they are higher dimensional and such dynamical
systems are important, but there are relatively few
dynamical phenomena that are currently understood \cite{De}, \cite{GMR}.

Each fixed point is an equilibrium state and
Proposition 2 gives a stable and unstable manifold which biologically means that if an initial point
(state) of the population is from the stable (resp. unstable) manifold then in future the state
of the population goes close (resp. far) to (resp. from) the state described by the fixed point.

As it was mentioned before, for each state $x\in \mathbb E_\epsilon$ of the mosquito population
its the next generation is given
by the evolution map $V(x)=x^2$ (see \cite{ly}) with respect to the multiplication (\ref{xx}).
Therefore, an absolute nilpotent element is a state of the population which dies in the next generation.
Proposition \ref{pan} says that the mosquito population has no any such (non-zero) state.
An idempotent element is a state of the population which does not change in the next generation, Proposition \ref{pie}
says that the population with baseline parameters (\ref{base}) has at least three
such states.

Ideals on evolution algebras have biological meaning that when the system
reach an ideal then all individuals in future generations will stay in it, i.e.
an ideal is closed subpopulation.

Under condition of Theorem \ref{ts} the mosquito population has not a closed subpopulation
(distinct from the full population itself). But in case of Theorem \ref{tns} there is a closed
subpopulation which has 5 generators.

Results of Proposition \ref{pl} and Proposition \ref{pu} show some sets of initial states the
population started from them tend to zero as time goes on, leading to
extinction of the population and some other initial states which lead the
population to become stable or arbitrarily large.

\section*{ Acknowledgements}

The work partially supported by Projects MTM2016-76327-C3-2-P and MTM2016-
79661-P of the Spanish Ministerio of Econom\'{\i}a and Competitividad, and Research Group FQM 199 of the Junta de Andaluc\'{\i}a (Spain),  all of them include European Union FEDER support; grant 853/2017 Plan Propio University of Granada (Spain);
 Kazakhstan Ministry of Education and Science, grant 0828/GF4.

{}
\end{document}